\documentclass[10pt]{amsart}

\usepackage{amsfonts}
\usepackage{amssymb}
\usepackage{graphics}
\usepackage{amsmath, amsthm, latexsym}
\usepackage[all]{xy}

\usepackage{hyperref}
\usepackage[margin=1.5in]{geometry}

\def \c{\mathbb{C}}
\def \z{\mathbb{Z}}

\def \p{\mathbb{P}}
\def \q{\mathbb{Q}}

\def \L{\mathcal{L}}

\def \K{{\bf K}}
\def \G{{\bf G}}

\def \X{\mathfrak{X}}

\def \D{\mathfrak{D}}
\def \k{{\bf k}}

\def \U{{\bf U}}

\def \dim{\textup{dim}}

\def \CDiv{\textup{CDiv}}

\def \Pic{\textup{Pic}}

\def \ratmap{\dashrightarrow}

\theoremstyle{plain}
\newtheorem{Th}{Theorem}[section]
\newtheorem{Lem}[Th]{Lemma}
\newtheorem{Prop}[Th]{Proposition}

\theoremstyle{definition}

\newtheorem{Def}[Th]{Definition}
\newtheorem{Rem}[Th]{Remark}

\begin{document}

\title{Note on the Grothendieck group of subspaces of rational functions and Shokurov's b-divisors}
\author{Kiumars Kaveh}
\address{Department of Mathematics, University of Pittsburgh,
Pittsburgh, PA, USA.}
\email{kaveh@pitt.edu}

\author{A. G. Khovanskii}
\address{Department of Mathematics, University of Toronto, Toronto,
Canada; Moscow Independent University; Institute for Systems Analysis, Russian Academy of Sciences.}
\email{askold@math.utoronto.ca}

\begin{abstract}
In a previous paper the authors develop an intersection theory for subspaces of rational functions on an algebraic variety $X$
over $\k = \c$. In this note, we first extend this intersection theory to an arbitrary algebraically closed ground field $\k$. 
Secondly we give an isomorphism between the group of b-divisors on the birational class of $X$ and the 
Grothendieck group of the semigroup of subspaces of rational functions on $X$. The constructed isomorphism moreover 
preserves the intersection numbers. This provides an alternative approach to b-divisors and their intersection theory.
\end{abstract}

\thanks{The first author is partially supported by a
Simons Foundation Collaboration Grants for Mathematicians (Grant ID: 210099) and a National Science Foundation Grant 
(Grant ID: 1200581).}

\thanks{The second author is partially supported by the Canadian Grant No. 156833-12.}


\date{\today}

\maketitle


\section*{Introduction}
In \cite{KKh-MMJ} the authors develop an intersection theory for subspaces of rational functions on an arbitrary variety over
$\k = \c$. In this short note we first extend this intersection theory to an arbitrary algebraically closed field $\k$, and secondly we observe that there is a direct connection between this intersection theory and Shokurov's b-divisors. This approach provides an alternative way of introducing b-divisors and their intersection theory, and in our opinion, is suitable for several applications in intersection theory.

Let $X$ be an irreducible variety of dimension $n$ over an algebraically closed ground field $\k$. Consider the collection $\K(X)$ of all the finite dimensional $\k$-subspaces of rational functions on $X$. The set $\K(X)$ is equipped with a natural product: for two subspaces $L, M \in \K(X)$, the product $LM$ is the subspace spanned by all the $fg$ where 
$f \in L$ and $g \in M$. With this product $\K(X)$ is a commutative semigroup (which may not be canecllative). Let $L_1, \ldots, L_n$ be subspaces in $\K(X)$, in \cite{KKh-MMJ} we associate a non-negative integer $[L_1, \ldots, L_n]$ to 
the subspaces $L_i$ and call it their intersection index. It is defined to be the number of solutions $x$ of a system 
$f_1(x) = \cdots = f_n(x) = 0$ where $f_i \in L_i$ are general elements and $x$ lies in a certain non-empty open subset $U$ of $X$ (depending on the $L_i$). In \cite{KKh-MMJ} it is shown that, when $\k = \c$, the intersection index is well-defined and moreover it is multi-additive with respect to the product of subspaces. It follows that the intersection index extends to a multi-additive integer valued function on the Grothendieck group $\G(X)$ of the semigroup $\K(X)$ (see Section \ref{sec-Grothendieck-group}). We regard $\G(X)$, together with its intersection index, as an extension of the intersection theory of (Cartier) divisors on complete varieties. In this note we observe that the the well-definedness and multi-additivity of the intersection index for an arbitrary algebraically
closed field $\k$ follows from the usual intersection theory on a product of projective spaces (see Sections \ref{sec-int-index} and 
\ref{sec-int-index-proofs}).

Consider the collection of all projective birational models of $X$, i.e. all birational maps 
$\pi: X_\pi \ratmap X$ where $X_\pi$ is projective. A (Cartier) b-divisor on $X$ is a direct limit of Cartier divisors $(X_\pi, D_\pi)$, with respect to a natural partial order on birational models of $X$. One verifies that the intersection product of Cartier divisors induces an intersection product on b-divisors (see Section \ref{sec-b-divisor}). The b-divisors (birational divisors) were introduced by Shokurov (see \cite{Iskovskikh, Shokurov}) and play an important role in 
birational geometry. 

The main observation of this note is that the Grothendieck group $\G(X)$ of $\K(X)$ can be identified with the group of b-divisors on $X$, and moreover this identification preserves the intersection index (Theorem \ref{th-main}).

A b-divisor is represented by a divisor on a projective birational model of the variety $X$. The collection of birational models of a variety, the main object of study in birational geometry, is a complicated object and intrinsically related to the notion of resolution of singularities and Minimal Model Program.  Moreover, proving statements about b-divisors and their intersection theory relies on the statements about usual divisors and their intersection theory, while the intersection theory of b-divisors is more stable in the sense that it is invariant under birational isomorphisms and one may regard it as easier to treat.

On the other hand, the Grothedieck group construction in \cite{KKh-MMJ}
suggests a different way which does not involve completions/birational models of $X$, and the 
invariance under a birational isomorphism is evident from the definition. This description of b-divisors and their intersection theory is suitable for several applications. We mention few here: (1) It provides a framework to extend the celebrated Bernstein-Kushnirenko theorem (from toric geometry) on the number of solutions of a system of Laurent polynomial equations in the algebraic torus $(\k^*)^n$, to arbitrary varieties and arbitrary systems of equations (see \cite{KKh-MMJ}). (2) Fix a valuation 
$v$ on the field of rational functions $\k(X)$ and with values in $\z^n$. One can then associate certain convex bodies (Newton-Okounkov bodies) to subspaces of rational functions, which using this approach can be identified with b-divisors, such that their Euclidean volumes give the intersection numbers of the corresponding b-divisors. This way one obtains transparent proofs of the Hodge inequality, and its generalizations, for intersection numbers \cite{KKh-Annals}.

Finally, we believe that the elementary nature of the notions needed to define the semigroup $\K(X)$ and its intersection index, 
i.e. subspaces of rational functions and number of solutions of a system of equations, makes this approach accessible to a wide range of audience and in particular suitable for a first course in algebraic geometry.

And few words about the organization of this short note:
Section \ref{sec-Grothendieck-group} covers basic definitions about the semigroup $\K(X)$ of subspaces of rational functions.
Sections \ref{sec-int-index} and \ref{sec-regular-Kodaira} recall some material about intersection index from \cite{KKh-MMJ}. 
Section \ref{sec-b-divisor} recalls basic definitions about b-divisors. In Section \ref{sec-main} we show that the Grothendieck group 
$\G(X)$ can be identified with the group of b-divisors. The last section is devoted to (short) proofs of the well-definedness and multi-additivity of the intersection index for an arbitrary algebraically closed field $\k$.



\section{Subspaces of rational functions and Grothendieck group} \label{sec-Grothendieck-group}
Throughout this note the ground field $\k$ is an algebraically closed field of arbitrary characteristic. Let $X$ be an 
algebraic variety over $\k$.

\begin{Def}
We denote the collection of all non-zero finite dimensional subspaces of the field of rational functions 
$\k(X)$ by $\K(X)$. Given $L, M \in \K(X)$ let the product $LM$ be the subspace spanned by all the products $fg$,
$f \in L$ and $g \in M$. With this product of subspaces $\K(X)$ is a commutative semigroup. 
\end{Def}

Let $K$ be a commutative semigroup (whose operation we denote by
multiplication). $K$ is said to have the {\bf cancellation property} if
for $x,y,z \in K$, the equality $xz=yz$ implies $x=y$. Any
commutative semigroup $K$ with the cancellation property can be extended
to an abelian group $G(K)$ consisting of formal quotients $x/y$, $x,
y \in K$. For $x,y,z,w \in K$ we identify the quotients $x/y$ and
$w/z$, if $xz = yw$.

Given a commutative semigroup $K$ (not necessarily with the cancellation
property), we can get a semigroup with the cancellation property by
considering the equivalence classes of a relation $\sim$ on $K$:
for $x, y \in K$ we say $x \sim y$ if there is $z \in K$ with $xz = yz$. The
collection of equivalence classes $K / \sim$ naturally has structure
of a semigroup with cancellation property. Let us denote the group
of formal quotients of $K / \sim$ again by $G(K)$. It is called the {\bf
Grothendieck group of the semigroup $K$}. The map which sends $x \in K$ to its
equivalence class $[x] \in K / \sim$ gives a natural homomorphism
$\phi: K \to G(K)$.

The Grothendieck group  $G(K)$ together with the homomorphism $\phi: K \to
G(K)$ satisfies the following universal property: for any other
group $G'$ and a homomorphism $\phi': K \to G'$, there exists a unique
homomorphism $\psi: G(K) \to G'$  such that $\phi' = \psi \circ
\phi$.

\begin{Def}
For two subspaces $L, M \in \K(X)$, we write $L \sim M$
if $L$ and $M$ are equivalent as elements of the multiplicative
semigroup $\K(X)$, that is, if there is $N \in \K(X)$
with $LN = MN$.
\end{Def}

Let $L \in \K(X)$. A rational function $f$ is said to be {\bf integral over $L$} if it satisfies an equation $f^n + \sum_{i=0}^{n-1}g_if^i = 0$, where $g_i \in L^{n-i}$, $i=1,\ldots, n-1$. The {\bf completion}
$\overline{L}$ is the collection of all rational functions which are integral
over $L$. The following result describes the completion $\overline{L}$ of a subspace $L$ as the
largest subspace equivalent to $L$ (see \cite[Appendix 4]{Zariski}
for a proof, also see \cite{KKh-MMJ}).
\begin{Th} \label{th-completion}
(1) The completion $\overline{L}$ is finite dimensional. (2) The completion $\overline{L}$ is the largest
subspace which is equivalent to $L$. That is, (a) $\overline{L}
\sim L$, and (b) if for $M \in \K(X)$ we have $M \sim
L$ then $M \subset \overline{L}$.
\end{Th}

\begin{Rem}
Let us call a subspace $L$, {\bf complete} if $\overline{L} = L$. If
$L$ and $M$ are complete subspaces, then $LM$ is not necessarily
complete. For two complete subspaces $L, M \in \K$,
define $$L * M = \overline{LM}.$$ The collection of complete
subspaces together with $*$ is a semigroup with the cancellation
property. Theorem \ref{th-completion} in fact shows that $L \mapsto
\overline{L}$ gives an isomorphism between the quotient semigroup
$\K / \sim$ and the semigroup of complete
subspaces (with $*$).
\end{Rem}

We denote the Grothendieck group of the semigroup $\K(X)$ by $\G(X)$.

\section{Intersection index of subspaces of rational functions} \label{sec-int-index}
In this section we define the intersection index of finite dimensional subspaces of rational function.
Let ${\bf L} = (L_1, \ldots, L_n)$ be an $n$-tuple of non-zero finite dimensional subspaces of rational functions.
Let $Z \subset X$ be a closed subvariety of $X$ cotaining the poles of all rational functions from the $L_i$, as well as 
all the points $x$ at which all functions from some subspace $L_i$ vanish. 

\begin{Th}[Intersection index is well-defined] \label{th-int-index-well-def}
There exists a non-empty Zariski open subset $\U \subset L_1 \times \cdots \times L_n$ such that 
for any $(f_1, \ldots, f_n) \in \U$ the number of solutions
$$\{ x \in X \setminus Z \mid f_1(x) = \cdots = f_n(x) = 0 \}$$ 
is finite and is independent of the choice of $(f_1, \ldots, f_n) \in \U$.
\end{Th}

We denote the number of solutions $\{ x \in X \setminus Z \mid f_1(x) = \cdots = f_n(x) = 0 \}$ in Theorem 
\ref{th-int-index-well-def} by $[L_1, \ldots, L_n]$ and call it the {\bf intersection index} of the subspaces $L_i$.

The following are immediate corollaries of the definition of the intersection index:
(1) $[L_1,\dots,L_n]$ is a symmetric function of
$L_1,\dots,L_n \in \K(X)$, (2) The intersection index
is monotone, (i.e. if $L'_1\subseteq L_1,\dots, L'_n\subseteq L_n$,
then $[L'_1,\dots,L'_n] \leq [L_1,\dots,L_n]$, and (3) The intersection index is non-negative.

\begin{Th}[Multi-additivity of intersection index] \label{th-int-index-multiadd}
Let $L_1', L_1'', L_2, \ldots, L_n \in {\bf K}(X)$
and put $L_1= L_1'L_1''$. Then
$$[L_1,\dots,L_n]=[L'_1,L_2,\dots,L_n]+[L''_1,L_2,\dots,L_n].$$
\end{Th}

We will prove Theorem \ref{th-int-index-well-def} and Theorem \ref{th-int-index-multiadd} in Section \ref{sec-int-index-proofs}.
The proofs relies on the notion of intersection product in the Chow rings of products of projective spaces.

From multi-additivity of the intersection index it follows that the
intersection index is invariant under the equivalence of subspaces,
namely if $L_1, \ldots, L_n$ and $M_1, \ldots, M_n \in \K$ are $n$-tuples of subspaces and for each $i$, $L_i
\sim M_i$ then
$$[L_1, \ldots, L_n] = [M_1, \ldots, M_n].$$ Hence one can extend the
intersection index to the Grothendieck group $\G(X)$ of $\K$. 
In particular, Theorem \ref{th-completion} implies that all the
intersection indices of a subspace $L \in \K$ and its
completion $\overline{L}$ are the same.

Analogous to the Kodaira map of a very ample line bundle, one can assign to a subspace $L \in \K(X)$ its Kodaira map $\Phi_L$ which is a rational 
map from $X$ to $\p(L^*)$, the projectivization of the dual space $L^*$:
let $x \in X$ be such that $f(x)$ is defined for all $f \in L$. Then $\Phi(x)$ is represented by the linear functional
in $L^*$ which sends $f$ to $f(x)$. 

Let $Y_L$ denote the closure of the image of $X$ in $\p(L^*)$. The next proposition relates the self-intersection index of a subspace with the degree of $Y_L$. It easily follows from the definition of the intersection index.
\begin{Prop}[Self-intersection index and degree] \label{prop-self-int-deg}
Let $L \in \K$ be a subspace and $\Phi_L: X \ratmap Y_L \subset \p(L^*)$ its
Kodaira map. (1) If $\dim X = \dim Y_L$ then $\Phi_L$ has finite mapping degree $d$ and
$[L, \ldots, L]$ is equal to the degree of the subvariety $Y_L$ (in $\p(L^*)$) multiplied with $d$.
(2) If $\dim X > \dim Y_L$ then $[L, \ldots, L] = 0$.
\end{Prop}

\section{Cartier divisor associated to a subspace of rational functions
with a regular Kodaira map} \label{sec-regular-Kodaira}  

The material in this Section are taken from \cite[Section 6]{KKh-MMJ}.

A {\it Cartier divisor} on an irreducible variety $X$ is a divisor which can
be represented locally as a divisor of a rational function. Any
rational function $f$ defines a {\it principal Cartier divisor}
denoted by $(f)$. The Cartier divisors are closed under the addition and
form an abelian group which we will denote by $\CDiv(X)$. A
dominant morphism $\Phi:X\to Y$ between varieties $X$ and $Y$ gives a pull-back
homomorphism $\Phi^*: \CDiv(Y)\to \CDiv(X)$. Two Cartier divisors are
linearly equivalent if their difference is a principle divisor. The
group of Cartier divisors modulo linear equivalence is called the
{\it Picard group} of $X$ and denoted by $\CDiv(X)$. One has an
intersection theory on $\Pic (X)$: for given Cartier divisors
$D_1,\dots, D_n$ on an $n$-dimensional complete variety  there is
an intersection index $[D_1,\dots,D_n]$ which obeys the usual
properties (see \cite{Fulton}).

Now let us return back to the subspaces of rational functions. For a
subspace $L\in \K(X)$, in general, the Kodaira map $\Phi_L$ is
a rational map, possibly not defined everywhere on $X$.

We denote the collection of subspaces $L\in \K(X)$
for which the rational
Kodaira map $\Phi_L:X \ratmap \p(L^*)$ extends to a regular map defined everywhere on $X$,
by $\K_{Cart}(X)$. We call a subspace $L \in \K(X)$ a {\bf subspaces with regular Kodaira map}.
One verifies that the collection $\K_{Cart}(X)$ is closed under
the multiplication. 

To a subspace $L \in \K_{Cart}(X)$ there naturally corresponds
a Cartier divisor $\mathcal{D}(L)$ as follows: each rational function $h\in L$
defines a hyperplane $H=\{ h=0\}$ in $\p(L^*)$. The divisor $\mathcal{D}(L)$ is
the difference of the pull-back divisor $\Phi ^*_L(H)$ and the
principal divisor $(h)$.

One proves the following:
\begin{Th} \label{th-DL-well-defined}
Let $X$ be a projective variety. Then:
(1) For any $L\in \K_{Cart}(X)$ the divisor $\mathcal{D}(L)$ is well-defined,
i.e. is independent of the choice of a function $h\in L$.
(2) The  map $L\mapsto \mathcal{D}(L)$ is a homomorphism from the semigroup
$\K_{Cart}(X)$ to the semigroup $\CDiv(X)$.
(3) The map $L\mapsto \mathcal{D}(L)$ preserves the intersection index, i.e.
for $ L_1,\dots, L_n\in \K_{Cart}(X)$ we have $$[L_1,\dots,L_n]=
[\mathcal{D}(L_1),\dots,\mathcal{D}(L_n)],$$ where the right-hand side is the
intersection index of Cartier divisors.
\end{Th}

Conversely, to any Cartier divisor we can associate a subspace of rational functions.
The subspace $\mathcal{L}(D)$ associated to a Cartier
divisor $D$ is the collection of all rational functions $f$ such that
the divisor $(f)+ D$ is effective (by definition $0 \in L(D)$).

The following well-known fact can be found in \cite[Chap. 2, Theorem
5.19]{Hartshorne}.
\begin{Th} \label{th-LD-finite}
When $X$ is projective
$\mathcal{L}(D)$ is finite dimensional.
\end{Th}

We record the following facts which are direct corollaries of the definition.
\begin{Prop} 
Let $L\in \K_{Cart} (X)$ and put $D = \mathcal{D}(L)$.
Then $L\subset \mathcal{L}(D)$ and $\mathcal{L}(D)\in \K_{Cart} (X)$.
\end{Prop}

\begin{Prop} \label{prop-LD-birat}
Let $\rho: X' \to X$ be a birational morphism. Let $D$ be a Cartier divisor on $X$. Then
$$\L(\rho^*(D)) = \rho^*(\L(D)).$$
\end{Prop}

One verifies that $D$ is a very ample divisor if
$\mathcal{D}(\mathcal{L}(D))=D$. It is a well-known fact that the group of Cartier divisors is generated by the 
very ample divisors (see \cite[Example 1.2.6]{Lazarsfeld}).

Let $L\in \K_{Cart}(X)$. The following describes the
subspace $\mathcal{L}(\mathcal{D}(L))$. It can be found in slightly different forms in
\cite[Chap. 2, Proof of Theorem 5.19]{Hartshorne} and \cite[Appendix 4]{Zariski}.
\begin{Th} \label{th-LDL}
Let $X$ be an irreducible projective variety and let $L\in
\K_{Cart} (X)$ be such that the Kodaira map $\Phi_L: X \to
\p(L^*)$ is an embedding. Then:
(1) every element of $\mathcal{L}(\mathcal{D}(L))$ is integral over $L$, i.e.
$L \subset \mathcal{L}(\mathcal{D}(L))\subset \overline{L}$ and hence $\mathcal{L}(\mathcal{D}(L)) \sim L$,
moreover (2) if $X$ is normal then $\mathcal{L}(\mathcal{D}(L))= \overline{L}.$
\end{Th}

\section{b-divisors} \label{sec-b-divisor}
Let $X$ be an irreducible variety of dimension $n$ defined over an algebraically closed field $\k$. 

\begin{Def}[Birational model] \label{def-birat-model}
We call a proper birational map $\pi: X_\pi \ratmap X$, where $X_\pi$ is a  projective variety, a 
{\bf projective birational model} (or for short a birational model) of $X$. The collection of all models of $X$ modulo isomorphism is partially ordered: 
We say $(X_{\pi'}, \pi')$ {\bf dominates} $(X_\pi, \pi)$ and write $\pi' \geq \pi$ if there is a morphism $\rho: X_{\pi'} \to X_{\pi}$ such that $\pi' = \pi \circ \rho$.
\end{Def}

\begin{Prop} \label{prop-birat-inductive}
The above partial order is inductive, i.e. for any two models $(X_\pi, \pi)$ and $(X_{\pi'}, \pi')$ there exists a third model
$(X_{\pi''}, \pi'')$ which dominates both. 
\end{Prop}
\begin{proof}
By Chow's lemma, without loss of generality we can assume that $X$ is quasi-projective sitting in some projective space $\p^N$.
Let $U \subset X$ be an open subset such that $\pi$ and $\pi'$ are isomorphisms restricted to $\pi^{-1}(U)$ and 
$\pi'^{-1}(U)$ respectively. Consider the set 
$$\Gamma = \{ (x, \pi^{-1}(x), \pi'^{-1}(x)) \mid x \in U\} \subset X \times X_\pi \times X_{\pi'},$$
and let $X_{\pi''}$ be the Zariski closure of $\Gamma$ in $\p^N \times X_\pi \times X_{\pi'}$.
The morphisms to $X_\pi$ and $X_{\pi'}$ as well as the rational map to $X$ are given by the projections on the corresponding factors.
\begin{equation} \label{equ-dominate}
\xymatrix{
& X_{\pi} \ar@{-->}[dr]^{\pi'} & \\
X_{\pi''} \ar[ur]^{\rho}  \ar@{-->}[rr]^{\pi''} \ar[dr]^{\rho'} & & X \\
& X_{\pi'} \ar@{-->}[ur]^{\pi'} & \\
}
\end{equation}
\end{proof}

\begin{Def}
The {\bf Riemann-Zariski space} $\X$ of the birational class of $X$ is defined as the limit
$$\X = \lim_{\leftarrow_\pi} X_\pi,$$
where the limit is taken over all the birational models of $X$. 
\end{Def}

\begin{Def}[b-divisor]
Following Shokurov one defines the group of {\bf Cartier b-divisors} as
$$\CDiv(\X) = \lim_{\rightarrow_\pi} \CDiv(X_\pi),$$ where $\CDiv(X_\pi)$ denotes the group of Cartier divisors
on the variety $X_\pi$ and the limit is taken with respect to the pull-back maps 
$\CDiv(X_\pi) \to \CDiv(X_{\pi'})$ which are defined whenever $\pi' \geq \pi$.
\end{Def}

\begin{Rem}
Let $D_1, \ldots, D_n$ be Cartier divisors on a model $X_\pi$ and suppose $\pi' \geq \pi$. One shows that 
the intersection number of the $D_i$ on $X_\pi$ is equal to the intersection number of the their pull-backs to
$X_{\pi'}$. This shows that the intersection number of b-divisors is well-defined.
\end{Rem}

\section{Main statement} \label{sec-main}
Let $X$ be an irreducible variety over $\k$. In this section we prove the main statement of this note that the group of 
b-divisors is naturally isomorphic to the Grothendieck group of $\K(X)$ (Theorem \ref{th-main}). The proof is based on the following easy lemma:

\begin{Lem} \label{lem-extend-Kodaira}
Let $L \in \K(X)$ be a non-zero finite dimensional subspace with the Kodaira rational map $\Phi_L: X \ratmap \p(L^*)$.
Then there exists a normal projective birational model $X_\pi$ of $X$ such that $\Phi_L \circ \pi$ extends to a regular map on 
the whole $X_\pi$. In other words, the Kodaira map of the subspace $\pi^*(L)$ extends to a regular map on 
$X_\pi$.
\end{Lem}
\begin{proof}
Every irreducible variety is birationally isomorphic to a projective variety. 
So without loss of generality we assume that $X$ is projective.
Let $U \subset X$ be an open subset such that ${\Phi_L}_{| U}$ is regular. Let $\Gamma \subset U \times \p(L^*)$
be the graph of ${\Phi_L}_{|U}$ and let $X_\pi$ be the closure of $\Gamma$ in $X \times \p(L^*)$. Let 
$\pi: X_\pi \to X$ denote the projection on the first factor. The map $\pi_{| \Gamma}: \Gamma \to U$
is an isomorphism with inverse $x \mapsto (x, \Phi_L(x))$. Thus $X_\pi$ is birationally isomorphic to $X$. Also the Kodaira map
$\Phi_L$ on $X$ extends to $\Phi_{\pi*(L)}$ which is the projection on the second factor and hence defined on the whole $X_\pi$.
If $X$ is not normal, replace $X_\pi$ with its normalization.
\end{proof}

From Theorem \ref{th-DL-well-defined} and Lemma \ref{lem-extend-Kodaira} we get the following.
\begin{Th} \label{th-main}
The group of b-divisors $\CDiv(\X)$ is naturally isomorphic to the Grothendieck group $\G(X)$ of $\K(X)$. 
Moreover, the isomorphism preserves the intersection index.
\end{Th}

\begin{Rem}
Sometimes it is customary to define b-divisors to be the direct limit of vector spaces of $\q$-divisors on birational models of $X$. 
With this definition the main result would assert that the $\q$-vector space of b-divisors is isomorphic to 
the $\q$-vector space $\G(X) \otimes \q$. 
\end{Rem}

\begin{proof}[Proof of Theorem \ref{th-main}]
Define a map $F$ from $\CDiv(\X)$ to the Grothendieck group of $\K(X)$ as follows. 
Let $\D$ be a b-divisor represented by a Cartier divisor $D_\pi$ on a birational model $X_\pi$. 
We know that $D_\pi$ can be written as a difference of two very ample divisors on $X_\pi$. 
Thus to define $F$ it is enough to define it on b-divisors which are represented by very ample divisors. So
without loss of generality we assume that $D_\pi$ is very ample.
By Theorem \ref{th-LD-finite} we know that the subspace $\L(D_\pi) \subset \k(X_\pi)$ associated to $D_\pi$ is finite dimensional. 
Define $$F(\D) = (\pi^{-1})^*(\L(D_\pi)).$$
Suppose $(X_{\pi'}, \pi')$ dominates $(X_\pi, \pi)$ with the corresponding morphism $\rho: X_{\pi'} \to X_\pi$. Then by
Proposition \ref{prop-LD-birat} we have
$$\L(\rho^*(D_\pi)) = \rho^*(\L(D_\pi)),$$
since $\rho$ is a birational isomorphism.
Thus $F$ is well-defined, i.e. is independent of the choice of the representative $(X_\pi, D_\pi)$ for $\D$.

Now suppose $\D$, $\D'$ are two b-divisors represented by very ample divisors 
$D_\pi$, $D_{\pi'}$ on two birational models $(X_{\pi}, \pi)$ and $(X_{\pi'}, \pi')$. By Proposition \ref{prop-birat-inductive}
we can find a third birational model $(X_{\pi''}, \pi'')$ dominating both. Now applying Theorem \ref{th-DL-well-defined}(2) to 
the pull-backs of $D_\pi$ and $D_{\pi'}$ to $X_{\pi''}$,
it follows that $F(\D + \D') = F(\D) F(\D')$, that is, $F$ is a homomorphism. 

Next we define an inverse map $G$ to $F$. 
Suppose $L \in \K(X)$. Then by Lemma \ref{lem-extend-Kodaira} there exists a normal projective model $X_\pi$
such that $\Phi_{\pi^*(L)}: X_\pi \ratmap \p(\pi^*(L)^*)$ extends to a regular map on the whole $X_\pi$ (which we again denote 
by $\Phi_{\pi^*(L)}$). Define $G(L)$ to be the element of $\CDiv(\X)$ represented by the divisor $\mathcal{D}(\pi^*(L))$
in the birational model $X_\pi$. Suppose $X_{\pi'}$ is another birational model such that $\Phi_{\pi'^*(L)}$ is regular.
By Proposition \ref{prop-birat-inductive} we can find a third model $\pi''$ which dominates both $\pi$ and $\pi'$. Now 
$\rho^*(\mathcal{D}(\pi^*(L)) = \rho'^*(\mathcal{D}(\pi'^*(L)) = \mathcal{D}(\pi''^*(L))$. Hence the class in $\CDiv(\X)$
represented by $\mathcal{D}(L)$ is independent of the choice of the model $X_\pi$ and the map $G$ is well-defined.
Finally, if $\D$ is represented by a very ample divisor we know that $G(F(\D)) = \D$, and also by Theorem \ref{th-LDL}, 
$L \subset F(G(L)) \subset \overline{L}$ and hence $F(G(L)) \sim L$. So $F$ and $G$ are inverses of each other and the 
proposition is proved.
\end{proof}

\section{Intersection index is well-defined and multi-additive} \label{sec-int-index-proofs}
In this section we prove Theorems \ref{th-int-index-well-def} and 
\ref{th-int-index-multiadd} using intersection product in the Chow ring of a product of projective spaces. A standard reference for 
Chow rings and their intersection product is \cite[Chapter 8]{Fulton}.

Let $X$ be an irreducible $n$-dimensional variety. 
Let ${\bf L} = (L_1, \ldots, L_n)$ be an $n$-tuple of non-zero finite dimensional subspaces of rational functions on $X$. 
For each $i$ let $\Phi_{L_i}: X \ratmap \p(L_i^*)$ denote the corresponding Kodaira rational map. Suppose 
$X$ is birationally embedded in some projective space $\p^N$ (Chow's lemma). Put $$P = \p^N \times \p(L_1^*) 
\times \cdots \times \p(L_n^*),$$ and consider the rational map
$\Phi_{\bf L}: X \ratmap P$ given by
\begin{equation} \label{equ-Phi-L}
\Phi_{\bf L}: x \mapsto (x, \Phi_{L_1}(x), \ldots, \Phi_{L_n}(x)).
\end{equation}
Let $Y_{\bf L}$ be the closure of the image of $X$ under the rational map $\Phi_{\bf L}$. The map $\Phi_{\bf L}$ is a 
birational isomorphism between $X$ and $Y_{\bf L}$.

\begin{proof}[Proof of Theorem \ref{th-int-index-well-def}]
For each $i$, let $H_i$ be a hyperplane in $\p(L_i^*)$ and let ${\bf H} = \p^N \times H_1 \times \cdots \times H_n$.
Then $H$ is a subvariety of $P$ of codimension $n$. We note that for different choices of the hyperplanes $H_i$ the cycles $[{\bf H}]$ are all rationally equivalent. From the definition of product in the Chow ring $A_*(P)$ we see that
the intersection index $[L_1, \ldots, L_n]$ is equal to the intersection number of the product of cycles $[Y_{\bf L}]$ and 
$[{\bf H}]$, and hence well-defined.
\end{proof}

\begin{proof}[Proof of Theorem \ref{th-int-index-multiadd}]
There is a natural surjective map $L_1' \otimes L_1'' \to L_1'L_1''$. This induces an embedding $(L_1'L_1'')^* 
\hookrightarrow L_1' \otimes L_1''$, and thus an embedding $\p((L_1'L_1'')^*) \hookrightarrow \p((L_1' \otimes L_1'')^*)$.
Consider the Segre map $\p(L_1'^*) \times \p(L_1'') \to \p(L_1'^* \otimes L_1''^*) \cong \p((L_1' \otimes L_1'')^*)$. One has 
a commutative diagram:
\begin{equation} \label{equ-tensor}
\xymatrix{& \p^N \times \p((L_1'L_1'')^*) \times \p(L_2^*) \times \cdots \times \p(L_n^*) \ar[d] \\
X \ar[r] \ar[dr] \ar[ur] & \p^N \times \p((L_1' \otimes L_1'')^*) \times \p(L_2^*) \times \cdots \times \p(L_n^*)\\
& \p^N \times \p((L_1'^*) \times \p(L_1''^*) \times \p(L_2^*) \times \cdots \times \p(L_n^*) \ar[u]}
\end{equation}
where the top vertical map is given by $\p((L'_1L''_1)^*) \hookrightarrow \p((L'_1 \otimes L''_1)^*)$ 
and the bottom vertical map is induced by the Segre map. 
Note that the image of $X$ in $\p((L_1' \otimes L_1'')^*)$ lies in $\p((L_1'L_1'')^*)$.  
Take $f \in L_1'$, $g \in L_1''$. Then $f, g$ define hyperplanes $H_f$, $H_g$ in $\p({L'_1}^*)$, $\p({L_1''}^*)$ respectively.
Moreover $f \otimes g \in L_1' \otimes L_1'' \cong (L_1'^* \otimes L_1''^*)^*$ defines a hyperplane $H$ in 
$\p((L_1' \otimes L_1'')^*)$ and hence in $\p((L_1'L_1'')^*)$. Also take hyperplanes $H_2, \ldots, H_n$ in 
$\p(L_2^*), \ldots, \p(L_n^*)$ respectively. We know from the proof of Theorem \ref{th-int-index-well-def} that:
$$[L_1', L_2, \ldots, L_n] = [Y_{\bf L'}] \cdot [H_f \times H_2 \times \cdots \times H_n],$$
$$[L_1'', L_2, \ldots, L_n] = [Y_{\bf L''}] \cdot [H_g \times H_2 \times \cdots \times H_n],$$
$$[L_1'L_1'', L_2, \ldots, L_n] = [Y_{\bf L}] \cdot [H \times H_2 \times \cdots \times H_n],$$
where $\cdot$ denotes the product in the corresponding Chow rings, and 
${\bf L'} = (L_1', L_2, \ldots, L_n)$, ${\bf L''} = (L_1'', L_2, \ldots, L_n)$, ${\bf L} = (L_1'L_1'', L_2, \ldots, L_n)$.
We note that the pull-back, under the Segre map, of the hyperplane $H$ to $\p(L_1'^*) \times \p(L_1''^*)$ 
coincides with $(H_f \times \p(L_1''^*)) + (\p(L_1'^*) \times H_g)$. This shows that 
$$[H \times H_2 \times \cdots \times H_n] = [H_f \times H_2 \times \cdots \times H_n] + 
[H_g \times H_2 \times \cdots \times H_n].$$ This finishes the proof.
\end{proof}


\begin{thebibliography}{99}



\bibitem[Fulton98]{Fulton} Fulton, W. {\it Intersection theory}. Second edition. Ergebnisse der Mathematik und ihrer Grenzgebiete. 3. Folge. A Series of Modern Surveys in Mathematics, 2. Springer-Verlag, Berlin, 1998.

\bibitem[Hartshorne77]{Hartshorne} Hartshorne, R. {\it Algebraic geometry}.
Graduate Texts in Mathematics, No. 52. Springer-Verlag, New York-Heidelberg, 1977.

\bibitem[Iskovskikh03]{Iskovskikh} Iskovskikh, V. A. {\it b-divisors and Shokurov functional algebras}. (Russian) Tr. Mat. Inst. Steklova 240 (2003), Biratsion. Geom. Linein. Sist. Konechno Porozhdennye Algebry, 8--20; translation in Proc. Steklov Inst. Math. 2003, no. 1 (240), 4--15 



\bibitem[Kaveh-Khovanskii10]{KKh-MMJ}
Kaveh, K.; Khovanskii, A. G. {\it Mixed volume and an extension of intersection theory of
divisors}. {Moscow Mathematical Journal} 10 (2010), no. 2, 343--375.
 
\bibitem[Kaveh-Khovanskii12]{KKh-Annals} Kaveh, K.; Khovanskii, A. G. {\it Newton-Okounkov bodies, semigroups of integral points,
graded algebras and intersection theory}. Annals of Mathematics 176 (2012), 1--54.
 

\bibitem[Lazarsfeld04]{Lazarsfeld} Lazarsfeld, R.
{\it Positivity in algebraic geometry. I. Classical setting: line bundles and linear series}.
Ergebnisse der Mathematik und ihrer Grenzgebiete. 3.
Folge. A Series of Modern Surveys in Mathematics, 48. Springer-Verlag, Berlin, 2004.

\bibitem[Samuel-Zariski60]{Zariski} Samuel, P.; Zariski, O.
{\it Commutative algebra}. Vol. II. Reprint of the 1960 edition.
Graduate Texts in Mathematics, Vol. 29.

\bibitem[Shokurov03]{Shokurov} Shokurov, V. V. {\it Prelimiting flips}. Tr. Mat. Inst. Steklova 240 (2003), Biratsion. Geom. Linein. Sist. 
Konechno Porozhdennye Algebry, 82--219; translation in Proc. Steklov Inst. Math. 2003, no. 1 (240), 75--213.
\end{thebibliography}
\end{document}